\documentclass[12pt]{amsart} 
\usepackage[mathscr]{eucal} 
\usepackage{amsmath,amsfonts} 
\parskip=\smallskipamount 
\newtheorem{theorem}{Theorem}[section] 
\newtheorem{proposition}[theorem]{Proposition} 
\newtheorem{corollary}[theorem]{Corollary} 
\newtheorem{lemma}[theorem]{Lemma} 
\theoremstyle{definition} 
 
\newtheorem{example}[theorem]{Example} 
\newtheorem{examples}[theorem]{Examples}
 
\newtheorem{remark}[theorem]{Remark} 

\makeatletter 
\@addtoreset{equation}{section} 
\makeatother

\makeatletter
\@namedef{subjclassname@2020}{\textup{2020} Mathematics Subject Classification}
\makeatother

\newcommand{\CC}{{\mathbb C}} 
\newcommand{\NN}{{\mathbb N}}

\newcommand{\cA}{{\mathcal A}} 
\newcommand{\cB}{{\mathcal B}} 
\newcommand{\cC}{{\mathcal C}} 
 
\newcommand{\cE}{{\mathcal E}} 
\newcommand{\cF}{{\mathcal F}} 
 
\newcommand{\cH}{{\mathcal H}} 
\newcommand{\cI}{{\mathcal I}}
\newcommand{\cJ}{{\mathcal J}} 
\newcommand{\cK}{{\mathcal K}} 
\newcommand{\cL}{{\mathcal L}}

\newcommand{\cR}{{\mathcal R}} 
\newcommand{\cS}{{\mathcal S}}

\newcommand{\fk}{\mathbf{k}}

\newcommand{\ra}{\rightarrow} 
\newcommand{\ol}{\overline}

\newcommand{\tr}{\operatorname{tr}} 
\let\phi=\varphi 
\newcommand{\iac}{\mathrm{i}}

\newcommand{\lin}{\operatorname{Lin}}

\newcommand{\tl}{\tilde}

\begin{document} 
\title[Automatic Boundedness of Adjointable Operators]{Automatic Boundedness of Adjointable Operators on  Barreled VH-Spaces}\thanks{}
 
 \date{\today}
 
 \author[S. Ay]{Serdar Ay}
 \address{Department of Mathematics, At{\.i}l{\.i}m University, 06830 {\.I}ncek, Ankara, 
Turkey}
 \email{serdar.ay@atilim.edu.tr}
 

\begin{abstract} 
We consider the space of adjointable operators on barreled VH (Vector Hilbert) spaces and show that such operators are automatically bounded. This generalizes the well known corresponding result for locally Hilbert $C^*$-modules. We pick a consequence of this result in the dilation theory of VH-spaces and show that, when barreled VH-spaces are considered, a certain boundedness condition for the existence of VH-space linearisations, equivalently, of reproducing kernel VH-spaces, is automatically satisfied.
\end{abstract} 

\maketitle 
 
\section*{Introduction}

VE (Vector Euclidean) and VH (Vector Hilbert) spaces were introduced by Loynes in 1965, in \cite{Loynes1} and \cite{Loynes2}. These are generalisations of the notion of an inner product space, where, the "inner product" has its values in a certain ordered vector space, called "ordered $*$-space", and otherwise has the similar properties with the usual inner product, see Section \ref{s:nps} for precise definitions. He was motivated by stochastic processes \cite{Loynes3} and he obtained a generalisation of B. Sz. Nagy Dilation Theorem \cite{BSzNagy} for operator valued positive semidefinite maps on $*$-semigroups along with other results about the spectral theory of order bounded linear operators of VH-spaces, see \cite{Loynes1} and \cite{Loynes2}. The notion of VH-spaces were used in prediction theory \cite{Weron}, \cite{WeronChobanyan}, as well as in dilation theory, \cite{GasparGaspar}, \cite{GasparGaspar2}, \cite{GasparGaspar1}, \cite{Gheondea}, \cite{AyGheondea1}, \cite{AyGheondea2}, \cite{AyGheondea3} and some others. 

VH-spaces can be considered as a rather general type of spaces and indeed, it is possible to consider VH-spaces over many different ordered $*$-spaces. This is the case of Hilbert modules over $C^*$-algebras, introduced in 1973 by Paschke in \cite{Paschke}, following \cite{Kaplansky}, and independently by Rieffel in \cite{Rieffel}; more generally, Hilbert modules over locally $C^*$-algebras, introduced in 1985 by Mallios in \cite{Mallios} and later by Phillips in \cite{Phillips} , also VH-spaces over $H^*$-algebras, considered by Saworotnow in \cite{Saworotnow}, are all special cases of a VH-space. 

Operator theory of VH-spaces has been considered by various authors. Loynes singled out and studied the $C^*$-algebra $\cB^*(\cH)$ of all order bounded adjointable linear operators on a VH-space $\cH$ in \cite{Loynes2}. In \cite{Gheondea}, a Stinespring type theorem was obtained for positive semidefinite kernels with values in $\cB^*(\cH)$, generalizing the Stinespring Theorem, \cite{Stinespring} in the VH-space direction. Given the generality of VH-spaces, it is clear that more general classes of linear operators must also be given consideration, and in \cite{Ciurdariu}, \cite{PaterBinzar}, \cite{AyGheondea2}, \cite{AyGheondea3} and some others, more general classes of continuous linear operators on VH-spaces were considered. Such operators found applications, in particular, in the dilation theory of invariant positive semidefinite kernels on a set under an action of a $*$-semigroup and valued in a topologically ordered $*$-space, see \cite{AyGheondea3}. 

On the other hand, operator theory on locally Hilbert $C^*$-modules was considered in \cite{Phillips}, \cite{Zhuraev}, \cite{Joita}, more recently in \cite{Gheondea2} and probably many others. The techniques to study these operators usually come from the corresponding techniques of Hilbert $C^*$-modules, but the arguments are not always straightforward.

One of the well known results of locally Hilbert $C^*$-modules states that the class of adjointable operators $\cL^*(\cH)$ on a locally Hilbert $C^*$-module $\cH$ consists of bounded operators and is a locally $C^*$-algebra with the locally convex Hausdorff topology given by the operator seminorms, see Subsection \ref{ss:lcalhm} for a review. In \cite{AyGheondea2} it was shown that this result is responsible for controlling one of the boundedness conditions of a dilation theorem for the existence of linearisations, or equivalently, reproducing kernel spaces of operator valued positive semidefinite kernels invariant under action of $*$-semigroups. This dilation theorem leads to a Kasparov type theorem in the context of locally Hilbert $C^*$-modules, generalizing the Hilbert $C^*$-module theorem in \cite{Kasparov}, first obtained in \cite{Joita}, and proven in a different way in \cite{AyGheondea2}, using VH-space dilations and following the ideas in \cite{Murphy}. An impetus to write this article was to understand whether a similar result holds for adjointable operators of VH-spaces, or more generally, of topological VE-spaces, and if so, what conditions must be imposed on them. We show that barreledness is a sufficient condition. A related notion of m-barreledness on certain $*$-algebras was considered in \cite{Haralam}.

From that point of view, Theorem \ref{t:adopbar} is the main result of this article. Main idea of the proof consists of passing to the quotient spaces by kernels of increasing $*$-seminorms and then employ the Uniform Boundedness Theorem on a certain family of vector valued functions on the quotient spaces. In addition, VH-space tools are used. The classical result above on locally Hilbert $C^*$-modules can be recovered by a variation of Theorem \ref{t:adopbar}, see Corollary \ref{c:quotbar} and the discussion in the end of Section \ref{s:adop}. 

We now briefly describe the contents of this article. In Section \ref{s:nps} we recall definitions and some properties of topologically ordered $*$-spaces, VE-spaces, VH-spaces and their operators. We prove Lemma \ref{l:clcone}, which shows that we can always assume that the cone of a topologically ordered $*$-space is closed. We recall two Schwarz type inequalities in Lemmas \ref{l:schwarz} and \ref{l:schwarzforposop} that will be used several times in the article. Subsection \ref{ss:lcalhm} briefly reviews locally $C^*$-algebras and locally Hilbert $C^*$-modules, mainly for convenience. 

In Section \ref{s:cclo} given two topological VE-spaces $\cE$ and $\cF$ over the same topologically ordered $*$-space $Z$, we study the class $\cL^*_{c}(\cE,\cF)$ of continuous and continuously adjointable operators and $\cL^*_{b}(\cE,\cF)$ of bounded and adjointable operators. In particular, if $\cE$ is a VH-space, then $\cL^*_{b}(\cE)$ is a locally $C^*$-algebra, see Corollary \ref{c:bddadjloc}.

The main result of this article, Theorem \ref{t:adopbar}, states that, an adjointable operator $T\colon \cE \ra \cF$ from a barreled topological VE-space $\cE$ to a topological VE-space $\cF$ over the same topologically ordered $*$-space $Z$ is automatically bounded, hence continuous. Here boundedness is in a rather strong sense that becomes the usual boundedness of linear operators when normed VE-spaces are considered, see the definitions in Section \ref{s:cclo}. A series of corollaries of the main result are proven, showing that the behaviour of adjointable operators of barreled topological VE-spaces are in a certain sense close to the adjointable operators of locally Hilbert $C^*$-modules, and in the normed case, to the adjointable operators of Hilbert $C^*$-modules. 

In the last section we turn to the dilation theory of VH-spaces and prove Theorem \ref{t:barvhinvkolmo}, as an application of Theorem \ref{t:adopbar}, which asserts that, when a barreled VH-space $\cH$ is considered, a boundedness condition (b2) from a characterization of the existence of invariant VH-space linearizations (equivalently, existence of reproducing kernel VH-spaces with $*$-representations) of positive semidefinite kernels valued in adjointable operators of $\cH$, see Theorem \ref{t:vhinvkolmo2}, is satisfied automatically. This provides another partial answer to a question raised in \cite{AyGheondea2}, asking when the condition (b2) in Theorem \ref{t:vhinvkolmo2} is satisfied automatically.

\section{Notation and Preliminary Results}\label{s:nps}

\subsection{Ordered $*$-Spaces.}\label{ss:as}
A complex vector space $Z$ is called an \emph{ordered $*$-space}, see e.g. \cite{Aliprantis} or 
\cite{PaulsenTomforde} if:
\begin{itemize}
\item[(a1)] $Z$ has an \emph{involution} $*$, that is, a map $Z\ni z\mapsto z^*\in Z$ 
that is \emph{conjugate linear} 
(($s x+t y)^*=\ol s x^*+\ol t y^*$ for all 
$s,t\in\CC$ and all $x,y\in Z$) and \emph{involutive} 
($(z^*)^*=z$ for all $z\in Z$). 
\item[(a2)] In $Z$ there is a \emph{convex cone} $Z_+$ ($s x+t y\in Z_+$ 
for all numbers $s,t\geq 0$ and all $x,y\in Z_+$), that is  
\emph{strict} ($Z_+\cap -Z_+=\{0\}$), and consisting of \emph{selfadjoint elements} 
only  ($z^*=z$ for all 
$z\in Z_+$). This cone is used to define a \emph{partial order} in selfadjoint elements $Z^{h}$ by: 
$z_1  \geq z_2$ if $z_1-z_2 \in Z_+$.
\end{itemize}

Note that, the set of all selfadjoint elements $Z^{h}$ is a real vector space. 

The complex vector space $Z$ is called a \emph{topologically ordered $*$-space} if it is 
an ordered $*$-space and:

\begin{itemize}
\item[(a3)] $Z$ is a \emph{Hausdorff separated locally convex space}.
\item[(a4)] The involution $*$ is \emph{continuous} with respect to this topology.
\end{itemize}
Condition (a4) is new compared to the definition in \cite{AyGheondea2} and \cite{AyGheondea3}.
\begin{itemize}
\item[(a5)] The cone $Z^+$ 
is \emph{closed} with respect to this topology. 
\item[(a6)] The topology of $Z$ is \emph{compatible} with the partial ordering in the 
sense that there exists a base of the topology, linearly generated by a family of 
neighbourhoods $\{N_j\}_{j\in\cJ}$ of the origin, such that all of them are 
absolutely convex and 
\emph{solid}, that is, whenever $x\in N_j$ and $0\leq y\leq x$ then $y\in N_j$.
\end{itemize}
It can be proven that axiom (a6) is equivalent with the following one, see \cite{AyGheondea2}:
\begin{itemize}
\item[(a6$^\prime$)] There exists a collection of seminorms 
$\{p_j\}_{j\in \cJ}$ defining the 
topology of $Z$ that are \emph{increasing}, that is, $0\leq x\leq y$ implies 
$p_j(x)\leq p_j(y)$.
\end{itemize}
We denote the collection of all increasing continuous seminorms on $Z$ by $S(Z)$. 
$Z$ is called an \emph{admissible space} if, in addition to the axioms (a1)--(a6),
\begin{itemize}
\item[(a7)] The topology on $Z$ is complete.
\end{itemize}

If the topology of a topologically ordered $*$-space $Z$ can be given by a single increasing norm, then we call $Z$ a \emph{normed ordered $*$-space}. If the topology on a normed ordered $*$-space $Z$ is complete, then $Z$ is called a \emph{Banach ordered $*$-space}.

\begin{remark}\label{r:add}
Notice that given two increasing seminorms $p$ and $q$ on a topologically ordered $*$-space $Z$, the seminorms $p+q$ 
and $\alpha p$ for any constant $\alpha \geq 0$ are also increasing. In addition, the maximum seminorm $r(z):=\mathrm{max}\{p(z),\,q(z)\}$, 
$z\in Z$ is increasing. In particular, the collection of all continuous and increasing seminorms $S(Z)$ is closed under addition, 
multiplication with a positive scalar and the maximum. 
\end{remark}

\begin{remark}\label{r:sasn}
Let us call a seminorm $p$ on a topologically ordered $*$-space a $*$-seminorm 
if $p(z^*)=p(z)$ for all $z\in Z$. 
It turns out that, the topology of a topologically 
ordered $*$-space can always be given by a family of continuous increasing $*$-seminorms. 
To see this, given a continuous increasing seminorm $q$, define $p(z):=\frac{1}{2}(q(z)+q(z^*))$ 
for any $z\in Z$. It is then easy to see that $p$ is a continuous $*$-seminorm. To 
show that it is increasing as well, let $0\leq z_1 \leq z_2$; in particular, 
$z_1=z_1^*$ and $z_2=z_2^*$. Then 
\begin{equation*}
p(z_1) = \frac{1}{2}(q(z_1)+q(z_1^*)) = q(z_1) \leq q(z_2) = \frac{1}{2}(q(z_1)+q(z_1^*)) = p(z_2) 
\end{equation*} 
and $p$ is increasing. 

Now given a collection of continuous increasing seminorms $\{q_j\}_{j\in \cJ}$ defining the topology 
of $Z$ as in $[(a5^\prime)]$, let $\{p_j\}_{j\in \cJ}$ be the corresponding collection 
of continuous increasing $*$-seminorms obtained by the above formula. We show that they 
give the same locally convex topology: Since the involution is continuous, 
for any $q_j$, there exists $\{q_{j_i}\}_{i=1}^{n}$ such that $q_j(z^*)\leq \sum_{i=1}^{n} q_{j_i}(z)$ 
for all $z$. Therefore $p_j(z)=\frac{1}{2}(q_j(z)+q_j(z^*)) \leq q_j(z) + \sum_{i=1}^{n} q_{j_i}(z)$. 
We also have $q_j(z) \leq 2 p_j(z)$ clearly and the equivalence of the topologies is shown. 
\end{remark} 

We denote the set of all continuous increasing $*$-seminorms by $S_*(Z)$. Similar to $S(Z)$, $S_*(Z)$ is 
closed under addition, multiplication with a positive scalar and the maximum. 

The following lemma is basically Corollary 2.24 in \cite{Aliprantis}. It shows that, closedness of the cone can always be assumed, by passing to the closure if necessary. 

\begin{lemma}\label{l:clcone}
Assume that an ordered $*$-space $Z$ satisfies all axioms (a3)--(a6) except for (a5), that is, the cone $Z^+$ is not necessarily closed with respect to the specified topology. Then by passing to the closure $\ol{Z^+}$ of the cone $Z^+$ we obtain a topologically ordered $*$-space. 
\end{lemma}

\begin{proof}
By passing to the closure $\ol{Z^+}$, we consider the space $Z$ with positive elements $\ol{Z^+}$. First we show that $(Z,\ol{Z^+})$ is an ordered $*$-space: It is clear that $\ol{Z^+}$ is closed under addition and multiplication with positive scalars. 

In order to see that $\ol{Z^+}$ consists of selfadjoint elements, let $z\in\ol{Z^+}$. Then there exists a net $(z_i) \subset Z^+$, $i\in \cI$ such that $z_i\ra z$. By continuity of $*$, $z_i^{*}\ra z^{*}$. As $z_i^*=z_i$ for any $i\in\cI$, it follows that $z_i\ra z^{*}$ and since the topology is Hausdorff, $z=z^*$.

We now show that $\ol{Z^+}$ is strict. Assume that there is $z\in\ol{Z^+}$ such that $-z\in\ol{Z^+}$ as well. Then there exist nets $(z_i)\subset Z^+$, $i\in\cI$ and $(w_j)\subset Z^+$, $j\in\cJ$ such that $z_i\ra z$ and $w_j\ra -z$. Let $p\in S_*(Z)$ be arbitrary and let $\epsilon > 0$. Pick $i_0\in\cI$ such that $p(z-z_i)< \epsilon/2$ for $i\geq i_0$ and also pick $j_0\in\cJ$ such that $p(w_j+z) < \epsilon/2$ for $j \geq j_0$. Observe that, by the increasing property of $p$, 
\begin{equation*}
0\leq p(z_i) \leq p(z_i + w_j) = p(z_i - z + z - w_j) 
\leq p(z-z_i) + p(z+w_j) < \epsilon
\end{equation*}
from which it follows that $p(z_i)\ra 0$. But we also have $p(z_i) \ra p(z)$, so $p(z)=0$. By the Hausdorff property and since $p$ was arbitrary, $z=0$ and therefore $\ol{Z^+}$ is strict. This completes the proof that $(Z,\ol{Z^+})$ is an ordered $*$-space.

Finally we show that any $p\in S(Z)$ is increasing when $\ol{Z^+}$ is considered instead of $Z^+$ as well: Let $y,z\in \ol{Z^+}$ be arbitrary and pick nets $z_i\ra z$ with $(z_i)\subset Z^+$, $i\in\cI$ and $y_j\ra y$ with $(y_j)\subset Z^+$, $j\in\cJ$. By increasing property of $p$, we have $p(z_i) \leq p(z_i + y_j)$ for all $i\in\cI$ and $j\in\cJ$. By passing to limits, it follows that $p(z) \leq p(z+y)$ and the increasing property is shown. In particular, any $p\in S_*(Z)$ is an increasing $*$-seminorm and the proof is complete.
\end{proof}
\subsection{Vector Euclidean Spaces and Their Linear Operators.}
Given a complex linear space $\cE$ and an
ordered $*$-space $Z$, a \emph{$Z$-valued inner product} or 
\emph{$Z$-gramian}  is, by definition, a mapping  
$\cE\times \cE\ni (x,y) \mapsto [x,y]\in Z$ subject to 
the following properties:
\begin{itemize}
\item[(ve1)] $[x,x] \geq 0$ for all $x\in \cE$, and $[x,x]=0$ if and only if $x=0$.
\item[(ve2)] $[x,y]=[y,x]^*$ for all $x,y\in\cE$.
\item[(ve3)] $[x,ay_1+by_2]=a[x,y_1]+b[x,y_2]$ for all $a,b\in \mathbb{C}$ and all
  $x_1,x_2\in \cE$.
\end{itemize}

A complex linear space $\cE$ onto which a $Z$-valued inner product 
$[\cdot,\cdot]$ is specified, for a 
certain ordered $*$-space $Z$, is called a \emph{VE-space} 
(Vector Euclidean space) over $Z$.  

In any VE-space $\cE$ over an ordered $*$-space $Z$ the familiar 
\emph{polarisation formula}
\begin{equation}\label{e:polar} 4[x,y]=\sum_{k=0}^3 \iac^k [(x+\iac^k y,x+\iac^k y],
\quad x,y\in \cE,
\end{equation} holds, which shows that the $Z$-valued inner product is perfectly 
defined by the $Z$-valued quadratic map $\cE\ni x\mapsto [x,x]\in Z$.

The concept of \emph{VE-spaces isomorphism} is also naturally defined: this is just 
a linear bijection $U\colon \cE\ra \cF$, for two VE-spaces over the same ordered $*$-space 
$Z$, which is \emph{isometric}, that is, 
$[Ux,Uy]_\cF=[x,y]_\cE$ for all $x,y\in \cE$. 

%

Given two VE-spaces $\cE$ and $\cF$, over the same ordered $*$-space $Z$, 
one can consider the vector space $\cL(\cE,\cF)$ of all
linear operators $T\colon \cE\ra\cF$. The operator $T$ is 
called \emph{order bounded} if there exists $C\geq 0$ such that
\begin{equation}\label{e:bounded} [Te,Te]_\cF\leq C^2 [e,e]_\cE,\quad e\in\cE.
\end{equation} Note that the inequality \eqref{e:bounded} is in the sense of the order 
of $Z$ uniquely determined by the cone $Z_+$, 
see the axiom (a2). 
The infimum of these scalars is denoted by $\|T\|$ and it is called 
the \emph{order operator norm} or \emph{order norm} of $T$, more precisely,
\begin{equation}\label{e:opnorm}\|T\|=\inf\{C>0\mid [Te,Te]_\cF\leq C^2 [e,e]_\cE,\mbox{ for all }e\in\cE\}.
\end{equation} Let $\cB(\cE,\cF)$ denote the collection of all order bounded linear operators $T\colon \cE\ra\cF$. Then $\cB(\cE,\cF)$ is a linear space 
and $\|\cdot\|$ is a norm on it, cf.\ Theorem 1 in \cite{Loynes2}. In addition, if $T$ 
and $S$ are order bounded linear operators acting between appropriate VE-spaces over 
the same ordered $*$-space $Z$, then $\|TS\|\leq \|T\| \|S\|$, in particular $TS$ is 
bounded. If $\cE=\cF$ then $\cB(\cE)=\cB(\cE,\cE)$ is a normed algebra, more 
precisely, the operator norm is submultiplicative. 

A linear operator $T\in\cL(\cE,\cF)$ is 
called \emph{adjointable} if there exists $T^*\in\cL(\cF,\cE)$ such that
\begin{equation}\label{e:adj} [Te,f]_\cF=[e,T^*f]_\cE,\quad e\in\cE,\ f\in\cF.
\end{equation} The operator $T^*$, if it exists, is uniquely determined by 
$T$ and called its \emph{adjoint}.
Since an analog of the Riesz Representation Theorem for VE-spaces
does not exist, in general, 
there may be not so many adjointable operators. We denote by 
$\cL^*(\cE,\cF)$ the vector space of all adjointable operators from
$\cL(\cE,\cF)$. 
Note that $\cL^*(\cE)=
\cL^*(\cE,\cE)$ is a $*$-algebra with respect to the involution $*$ determined 
by the operation of taking the adjoint.

An operator $A\in\cL(\cE)$ is called \emph{selfadjoint} if
\begin{equation}\label{e:self} [Ae,f]=[e,Af],\quad e,f\in \cE.
\end{equation} Clearly, any selfadjoint operator $A$ is adjointable and $A=A^*$.
By the polarisation formula \eqref{e:polar}, $A$ is selfadjoint if and only if
\begin{equation}\label{e:self2} [Ae,e]=[e,Ae],\quad e\in\cE.
\end{equation}

An operator $A\in\cL(\cE)$ is \emph{positive} if
\begin{equation}\label{e:pos} [Ae,e]\geq 0,\quad e\in\cE.\end{equation}
Since the cone $Z_+$ consists of selfadjoint elements only, 
any positive operator is selfadjoint and hence adjointable.

Let $\cB^*(\cE)$ denote the collection of all adjointable order bounded linear operators 
$T\colon \cE\ra\cE$. Then $\cB^*(\cE)$ is a pre-$C^*$-algebra, that is, it is a 
normed $*$-algebra with the property
\begin{equation}\label{e:prec} \|A^*A\|=\|A\|^2,\quad A\in\cB^*(\cE),
\end{equation} cf.\ Theorem 4 in \cite{Loynes2}.
 In particular, the involution $*$ is isometric on $\cB^*(\cE)$, that is, $\|A^*\|=\|A\|$ 
 for all $A\in\cB^*(\cE)$.
 
 If $A\in\cB^*(\cE)$ can be factored $A=T^*T$, for some $T\in\cB^*(\cE)$, 
 then $A$ is  positive. If, in addition, $\cB^*(\cE)$ is complete, and hence 
 a $C^*$-algebra, and $A\in\cB^*(\cE)$ 
is positive, then $A=T^*T$ for some $T\in\cB^*(\cE)$. This is the case when $\cE$ is a VH-space, see subsection \ref{ss:vhs}, by Lemma 2 in \cite{Loynes2}. 

\subsection{Vector Hilbert Spaces and Their Linear Operators.}\label{ss:vhs}
If $Z$ is a topologically ordered $*$-space, 
any VE-space $\cE$ can be made in a natural way
into a Hausdorff separated locally convex space by considering the weakest locally 
convex topology on $\cE$ that makes the mapping 
$E\ni h\mapsto [h,h]\in Z$ continuous, more 
precisely, letting $\{N_j\}_{j\in\cJ}$ be the collection of convex and solid 
neighbourhoods of the origin in $Z$ as in axiom (a5), the collection of sets
\begin{equation}\label{e:ujex}U_j=\{x\in \cE\mid [x,x]\in N_j\},\quad j\in\cJ,
\end{equation} is a topological base of neighbourhoods of the origin of $\cE$ that 
linearly generates the weakest locally convex topology on $\cE$ that 
makes all mappings $\cE\ni h\mapsto [h,h] \in Z$ continuous, cf.\ 
Theorem 1 in \cite{Loynes1}. In terms of seminorms, this topology can 
be defined in the following way: let $\{p_j\}_{j\in \cJ}$ be a family of increasing 
seminorms defining the topology of $Z$ as in axiom (a5$^\prime$) and let
\begin{equation}\label{e:qujeh} \tl{p_j}(h)=p_j([h,h])^{1/2},\quad h\in \cE,\ j\in\cJ.
\end{equation}
Then each $\tl{p_j}$ is a seminorm on $\cE$ and its topology 
is fully determined by the family $\{\tl{p_j}\}_{j\in\cJ}$, see Lemma~1.3 in \cite{AyGheondea2}. 
With respect to this topology, we call $\cE$ a 
\emph{topological VE-space} over $Z$. In the special case when $Z$ is a normed ordered $*$-space, we call $\cE$ a \emph{normed VE-space} over $Z$. 

The following is a Schwarz type inequality that we will use in the rest of the article. For a proof, see Lemma 2.2 of \cite{AyGheondea3}. 

\begin{lemma}\label{l:schwarz}
Let $\cE$ be a topological VE-space over the topologically ordered $*$-space 
$Z$ and $p\in S(Z)$. Then
\begin{equation}\label{e:schwarz}
p([e,f])\leq 4p([e,e])^{1/2}p([f,f])^{1/2}=4\tl{p}(e)\tl{p}(f), \quad e,f\in\cE.
\end{equation}
\end{lemma}

It turns out that if positive operators are considered, a stronger Schwarz type inequality holds. For 
a proof of the following lemma, see Lemma 2.13 of \cite{AyGheondea2}. 

\begin{lemma}\label{l:schwarzforposop}
Let $T \in \cL^*(\cE)$ be a positive operator on a topological VE-space $\cE$ 
over the topologically ordered $*$-space $Z$. Let $p\in S(Z)$. Then 
\begin{equation*}
p([Th,h] )\leq p([Th,Th] )^{\frac{1}{2}}
p([h,h] )^{\frac{1}{2}},\quad
h\in\cH.
\end{equation*} 
\end{lemma}

If both of $Z$ and $\cE$ are complete with respect to their specified
locally convex topologies, then $\cE$ is called a \emph{VH-space} 
(Vector Hilbert space), sometimes also called a \emph{pseudo Hilbert space}. In the case when $Z$ is normed, then $\cE$ is called a \emph{normed VH-space}. Any topological 
VE-space $\cE$ on a topologically ordered $*$-space 
can be embedded as a dense subspace of a VH-space $\cH$, 
uniquely determined up to an isomorphism, cf.\ Theorem 2 in
\cite{Loynes1}. Note that, given two VH-spaces $\cH$ and $\cK$, over the same 
admissible space $Z$, any isomorphism $U\colon \cH\ra\cK$ 
in the sense of VE-spaces, is automatically bounded and adjointable, hence 
$U\in\cB(\cH,\cK)$, and it is natural to call this operator \emph{unitary}.

Now we provide some examples of admissible spaces, topological VE-spaces and VH-spaces. For a more comprehensive list of examples, we refer to \cite{AyGheondea2}. 
\begin{examples}\label{e:evhs}

(1) Any $C^*$-algebra $\cA$ is an admissible space, as well 
as any closed $*$-subspace $\cS$ of a $C^*$-algebra $\cA$, with the 
positive cone $\cS^+=\cA^+\cap\cS$ and all other operations (addition, 
multiplication with scalars, and involution) inherited from $\cA$.

(2) Any pre-$C^*$-algebra is a topologically ordered $*$-space. Any 
$*$-subspace $\cS$ of a pre-$C^*$-algebra $\cA$ is a topologically ordered 
$*$-space, with the positive cone $\cS^+=\cA^+\cap\cS$ and all other 
operations inherited from $\cA$. Note that, by Lemma \ref{l:clcone} we can pass to the closure of $\cS^+$ if needed. 

(3) Any locally $C^*$-algebra, cf.\ \cite{Inoue}, \cite{Phillips}, 
see subsection \ref{ss:lcalhm}
 is an admissible 
space. 
In particular, any closed $*$-subspace $\cS$ of a locally $C^*$-algebra $\cA$, 
with the cone $\cS_+=\cA^+\cap\cS$ and all other operations inherited from 
$\cA$, is an admissible space. 

(4) Any locally pre-$C^*$-algebra, see subsection \ref{ss:lcalhm}, is a topologically ordered $*$-space. Any
$*$-subspace $\cS$ of a locally pre-$C^*$-algebra is a topologically ordered
$*$-space, with $\cS^+=\cA^+\cap\cS$ and all other operations inherited from 
$\cA$. Again, by Lemma \ref{l:clcone} we can pass to the closure of $\cS^+$ if needed.

(5) Any Hilbert $C^*$-module $\cH$ over a $C^*$-algebra $\cA$, see e.g. \cite{Lance} and \cite{ManuilovTroitsky}, as well as any closed linear subspace $\cS$ of a Hilbert $C^*$-module $\cH$ over the same $C^*$-algebra $A$ is an example of a normed VH-space. A pre-Hilbert $C^*$-module over a pre-$C^*$-algebra, as well as any linear subspace of a pre-Hilbert $C^*$-module, is an example of a normed  VE-space. 

(6) More generally, any locally Hilbert $C^*$-module $\cH$ over a locally $C^*$-algebra $\cA$, see subsection \ref{ss:lcalhm}, as well as any closed linear subspace $\cS$ of a locally Hilbert $C^*$-module over the same locally $C^*$-algebra $\cA$ is an example of a VH-space. A pre-locally Hilbert $C^*$-module over a pre-locally $C^*$-algebra, as well as any linear subspace of a pre-locally Hilbert $C^*$-module, is an example of a topological VE-space. 

(7) Let $\cH$ be an infinite dimensional separable Hilbert space and let
$\cC_1$ be 
the trace-class ideal, that is, the collection of all linear bounded operators
$A$ 
on $\cH$ such that $\tr(|A|)<\infty$. $\cC_1$ is a $*$-ideal of $\cB(\cH)$ and
complete under the 
norm $\|A\|_1=\tr(|A|)$. Positive elements in $\cC_1$ are defined in the sense 
of positivity in $\cB(\cH)$. In addition, the norm $\|\cdot\|_1$ is
increasing, since 
$0\leq A\leq B$ implies $\tr(A)\leq \tr(B)$, hence $\cC_1$ is a normed
admissible space.

(8) With notation as in (7), consider $\cC_2$ the ideal of Hilbert-Schmidt
operators on $\cH$. Then $[A,B]=A^*B$, for all $A,B\in\cC_2$, is a gramian with values in the 
admissible space $\cC_1$ with respect to which $\cC_2$ becomes a VH-space. Since $\cC_1$ is a normed admissible space, $\cC_2$ is a normed VH-space, with norm 
$\|A\|_2=\tr(|A|^2)^{1/2}$, for all $A\in\cC_2$, cf. \eqref{e:qujeh}. More abstract versions of this example have been
considered by Saworotnow in \cite{Saworotnow}.

(9) Let $\cH$ be a Hilbert space and $\cE$ a
VH-space over the admissible space $Z$. On the algebraic tensor product
$\cH\otimes \cE$ define a gramian by
\begin{equation*} [h\otimes e,l\otimes f]_{\cH\otimes\cE}
=\langle h,l\rangle_\cH [e,f]_\cE\in Z,\quad
  h,l\in \cH,\ e,f\in \cE,
\end{equation*} and then extend it to $\cH\otimes \cE$ by linearity. It can
be proven that, in this way, $\cH\otimes \cE$ is a VE-space over $Z$. Since
$Z$ is an admissible space, $\cH\otimes\cE$ can be topologised as in \eqref{e:qujeh} and then completed to a VH-space
$\cH\widetilde\otimes\cE$ over $Z$.
\end{examples}

\subsection{Locally $C^*$-Algebras and Locally Hilbert 
$C^*$-Modules.}\label{ss:lcalhm}
In this subsection we recall the definitions of locally $C^*$-algebras 
and locally Hilbert $C^*$-modules, and review some known facts.  

A $*$-algebra $\cA$ that has a complete 
Hausdorff topology induced by a family of
\emph{$C^*$-seminorms}, that is, seminorms $p$ on $\cA$ that satisfy the
\emph{$C^*$-condition} $p(a^*a) = p(a)^2$ for all $a\in\cA$, is called a 
\emph{locally $C^*$-algebra} \cite{Inoue} 
(equivalent names are \emph{(Locally Multiplicatively Convex) $LMC^*$-algebras}
\cite{Schmudgen}, \cite{Mallios}, or \emph{$b^*$-algebra}
\cite{Allan}, \cite{Apostol}, 
or \emph{pro $C^*$-algebra} \cite{Voiculescu}), \cite{Phillips}. Note that, any
$C^*$-seminorm is \emph{submultiplicative},  
$p(ab)\leq p(a)p(b)$ for all $a,b\in\cA$, cf.\ \cite{Sebestyen},
and is a $*$-seminorm, $p(a^*)=p(a)$ for all $a\in\cA$. 
Denote the collection of all continuous $C^*$-seminorms 
by $S^*(\cA)$. Then $S^*(\cA)$ is a 
directed set under pointwise maximum seminorm, namely, 
given $p,q\in S^*(\cA)$, letting 
$r(a):=\mbox{max}\{p(a),q(a)\}$ for all $a\in\cA$, then 
$r$ is a continuous $C^*$-seminorm and $p,q\leq r$. 
Locally $C^*$-algebras were studied in \cite{Allan}, \cite{Apostol},
\cite{Inoue}, \cite{Schmudgen}, \cite{Phillips}, and \cite{Zhuraev}, to cite a
few.  

It follows from Corollary 2.8 in \cite{Inoue} that any locally $C^*$-algebra 
is, in particular, an admissible space, more precisely, 
a directed family of increasing 
seminorms generating the topology in axiom (a5$^\prime$) 
in Subsection \ref{ss:as} is $S^*(\cA)$.
Note that $S^*(\cA) \subset S_*(\cA)$ and, although 
they generate the same topology on $\cA$, 
these two sets are quite different. For instance, 
while $S_*(\cA)$ is a cone, $S^*(\cA)$ is not even stable 
under positive scalar multiplication. 
 
A \emph{pre-Hilbert module over a locally $C^*$-algebra $\cA$}, 
or a \emph{pre-Hilbert $\cA$-module} is a topological VE-space $\cH$ 
over $\cA$ where, in addition we have a right module action of $\cA$ on $\cH$ that respects the gramian, namely, 
for every $a\in\cA$ and $h,g\in\cH$ we have $[g,ha]=[g,h]a$
. Note that the topology on the pre-Hilbert $\cA$-module $\cH$ is 
given by the family of seminorms $\{\tilde p\}_{p\in S^*(\cA)}$, where
$\tilde p(h)=p([h,h])^{1/2}$ for all $p\in S^*(\cA)$ and all $h\in\cH$. 
A pre-Hilbert $\cA$-module $\cH$ is called a 
\emph{Hilbert $\cA$-module} if it is complete, e.g.\ see \cite{Phillips} as well as \cite{Gheondea2}.

Let $\cH$ be a pre-Hilbert $\cA$-module, let $p \in S^*(\cA)$ 
and let $x,y \in \cH$. Then a Schwarz type inequality holds, 
e.g.\ see \cite{Zhuraev}, as follows
\begin{equation}\label{e:schwarz2}
p([ h,k]_{\cH})
\leq p([h,h]_{\cH})^{1/2}\, p([k,k]_{\cH})^{1/2},\quad h,k\in\cH.
\end{equation}
It is interesting to compare this inequality to the Schwarz type inequality of Lemma \ref{l:schwarz}, which is weaker but works for the larger class of seminorms $S(\cA)$.

\section{Classes of Continuous Linear Operators on Topological VE-Spaces}\label{s:cclo}

In this section, we study the properties of some classes of linear continuous operators 
acting between topological VE-spaces over the same topologically ordered $*$-space. 

For topological VE-spaces $\cE$ and $\cF$ over the same 
topologically ordered $*$-space $Z$, we denote 
the space of all linear continuous operators $T\colon\cE\ra\cF$ by $\cL_{\mathrm{c}}(\cE,\cF)$, 
and in particular, $\cL_{\mathrm{c}}(\cE,\cE)$ by $\cL_{\mathrm{c}}(\cE)$. By Remarks \ref{r:add} and \ref{r:sasn}, 
a characterization of elements of $\cL_{c}(\cE,\cF)$ is as follows: Fix an arbitrary family of generating seminorms $S_0(Z)\subseteq S_*(Z)$. Then a linear operator $T\colon \cE\ra \cF$ is in $\cL_{c}(\cE,\cF)$ if given any $p\in S_0(Z)$, 
there exists $q\in S_*(Z)$ with $\tl{p}(Tx) \leq \tl{q}(x)$ for all $x\in \cH$. It is easy to see that the choice of $S_0(Z)$ is immaterial; in particular, we can always choose $S_0(Z)$ to be $S_*(Z)$. The 
space of all continuous and continuously adjointable linear operators 
$T\colon\cE\ra\cF$ are denoted by $\cL^*_{\mathrm{c}}(\cE,\cF)$, and 
$\cL^*_{\mathrm{c}}(\cE)=\cL^*_{\mathrm{c}}(\cE,\cE)$, where the latter is a $*$-algebra, and in fact even an ordered $*$-algebra, see \cite{AyGheondea2} for the definition and a more detailed account.

For $\cE$ and $\cF$ as above, we denote by $\cL_{b}(\cE,\cF,S_0(Z))$ the set 
of linear and \emph{$S_0(Z)$-bounded} operators; this is 
the subset of $\cL_{c}(\cE,\cF)$ consisting of operators for which the seminorm $\tl{q}$ can be chosen 
to be $C \tl{p}$, where $S_0(Z)\subseteq S_*(Z)$ is a fixed generating family of seminorms, $C \geq 0$ is a constant dependent only on $T$ and $p$; i.e. we have, given $T\in \cL_{b}(\cE,\cF,S_0(Z))$ that, given any $p\in S_0(Z)$, there exists a constant $C\geq 0$ such that $\tl{p}(Tx)\leq C \tl{p}(x)$ for all $x\in\cE$. Let $\ol{p}(T)$ be the infimum of all such $C \geq 0$. Then 
it is routine to see that $\ol{p}$ is a seminorm on $\cL_{b}(\cE,\cF,S_0(Z))$ and that $\cL_{b}(\cE,\cF,S_0(Z))$ with the family of seminorms $\{\ol{p} \mid p\in S_0(Z) \}$ is turned into a Hausdorff locally convex space. In particular, we denote by $\cL_b(\cE,\cF)$ the space $\cL_b(\cE,\cF,S_*(Z))$, this is the space of all \emph{bounded} linear operators. 

Finally, we define $\cL_{b}^{*}(\cE,\cF,S_0(Z)):=\cL_{b}(\cE,\cF,S_0(Z)) \cap \cL^*(\cE,\cF)$, the space of all \emph{$S_0(Z)$-bounded and adjointable} operators. 
We use $\cL_{b}^{*}(\cE, S_0(Z))$ for $\cL_{b}^{*}(\cE,\cE, S_0(Z))$. In particular, we denote by $\cL_{b}^{*}(\cE,\cF):=\cL_{b}(\cE,\cF) \cap \cL^*(\cE,\cF)$, and call it  the space of all \emph{bounded and adjointable} operators. 

From now on, throughout the article, by $S_0(Z)$ we denote any fixed generating family of seminorms of a given topologically ordered $*$-space $Z$ with $S_0(Z)\subseteq S_*(Z)$.

\begin{lemma}\label{l:autobd}
Let $\cE$ and $\cF$ be VE-spaces over the topologically ordered $*$-space $Z$. Let $T\in \cL_{b}^{*}(\cE,\cF,S_0(Z))$. Then $T^*\in \cL_{b}^{*}(\cF,\cE,S_0(Z))$, and for any $p\in S_0(Z)$ we have $\ol{p}(T)=\ol{p}(T^*)$. 
\end{lemma}

\begin{proof}
Let $T\in \cL_{b}^{*}(\cE,\cF)$, $p\in S_0(Z)$ and $y\in \cF$. We have 
\begin{align*}
&p([T^*y,T^*y])=p([TT^*y,y]) \leq \tl{p}(TT^*y)\tl{p}(y) \\
&\leq \ol{p}(T)\tl{p}(T^*y)\tl{p}(y)
\end{align*} 
where for the first inequality Lemma \ref{l:schwarzforposop} is used. A standart argument now implies that $\ol{p}(T)=\ol{p}(T^*)$. 
\end{proof}

It follows from the preceeding lemma that, $\cL_{b}^{*}(\cE,\cF,S_0(Z))$ is a Hausdorff locally convex $*$-space; i.e. it is a $*$-space whose locally convex topology is given by $*$-seminorms and is Hausdorff. 

\begin{proposition}\label{p:bddcomp}
Let $\cF$ be a VH-space and $\cE$ be a topological VE-space over the admissible space $Z$. Then $\cL_{b}(\cE,\cF,S_0(Z))$ is a complete Hausdorff locally convex space. 
\end{proposition}

\begin{proof}
We only have to prove completeness and the proof goes using the same steps as in the corresponding proof for Banach spaces. For completeness, we briefly discuss the details. Let $(T_i)_{i\in \cI}\subset \cL_b(\cE,\cF,S_0(Z))$ be a Cauchy net. Then for any $p\in S_0(Z)$ and $\epsilon > 0$, there exists $i_0\in \cI$ such that $\ol{p}(T_i - T_j) < \epsilon$ for all $i,j > i_0$, therefore $\tl{p}(T_ix - T_jx)\leq \epsilon \tl{p}(x)$ for any $x\in\cE$. It follows that $(T_ix)_{i\in\cI}$ is a Cauchy net in $\cF$ for each $x\in\cE$, so $T_ix \ra y$ for some $y\in\cF$ by completeness of $\cF$. Let $T\colon \cE \ra \cF$ be the linear operator $Tx:=y$. 
For any $x\in\cE$, taking the limit over $j$ in the second inequality above, by continuity of $\tl{p}$ we obtain $\tl{p}(T_ix - Tx)\leq \epsilon \tl{p}(x)$ and hence $T_i - T \in \cL_b(\cE,\cF,S_0(Z))$ for $i > i_0$. Consequently $T=T_i - (T_i - T) \in \cL_b(\cE,\cF,S_0(Z))$ and clearly $T_i \ra T$ in $\cL_b(\cE,\cF,S_0(Z))$, which finishes the proof. 
\end{proof}

\begin{proposition}\label{p:bddadjcomp}
Let $\cF$ be a VH-space and $\cE$ be a topological VE-space over the admissible space $Z$. Then $\cL_{b}^{*}(\cE,\cF,S_0(Z))$ is a complete Hausdorff locally convex $*$-space. 
\end{proposition}

\begin{proof}
We only have to show completeness. For this, we show that, $\cL_{b}^{*}(\cE,\cF,S_0(Z))$ is closed in $\cL_{b}(\cE,\cF,S_0(Z))$. Let $(T_i)_{i\in\cI}\in\cL_{b}^{*}(\cE,\cF,S_0(Z))$ be a convergent net with $T_i\ra T$, $T\in\cL_{b}(\cE,\cF,S_0(Z))$. For any $p\in S_0(Z)$ since $\ol{p}(T_i^*)=\ol{p}(T_i)$ for all $i\in\cI$ by Lemma \ref{l:autobd}, $(T_i^*)_{i\in\cI}$ is a Cauchy net in $\cL_{b}(\cE,\cF,S_0(Z))$. By completeness of $\cL_{b}(\cE,\cF,S_0(Z))$ from Proposition \ref{p:bddcomp}, there exists $W\in \cL_{b}(\cE,\cF,S_0(Z))$ such that $T_i^* \ra W$. Then we have, for any $x\in\cE$ and $y\in\cF$
\begin{align*}
[Tx,y] &= \lim_{i} [T_ix,y] \\
&=\lim_{i} [x,T_i^*y] = [x,Wy]
\end{align*}
as $T_ix\ra Tx$ for each $x$, $T_i^*y\ra Wy$ for each $y$, $S_0(Z)$ is generating, and the gramian is continuous. Consequently, $[Tx,y]=[x,Wy]$ for all $x\in\cE$ and $y\in \cF$. Therefore $T$ is adjointable with $T^*=W$, completing the proof.
\end{proof}

\begin{proposition}\label{p:bddadjpreloc}
Let $\cE$ be a VE-space over the topologically ordered $*$-space $Z$. Then $\cL_{b}^{*}(\cE,S_0(Z))$ is a pre-locally $C^*$-algebra.
\end{proposition}

\begin{proof}
$\cL_{b}^{*}(\cE,S_0(Z))$ is a Hausdorff locally convex $*$-space by Lemma \ref{l:autobd}. For any $p\in S_0(Z)$ and $T_1, T_2 \in \cL_{b}^{*}(\cE,S_0(Z))$ it is easy to see that we have $\ol{p}(T_1T_2) \leq \ol{p}(T_1)\ol{p}(T_2)$, and it follows that $\cL_{b}^{*}(\cE,S_0(Z))$ is, in particular, a $*$-algebra. In addition, for any $p\in S_0(Z)$, $T\in \cL_{b}^{*}(\cE,S_0(Z))$ and $x\in\cE$ we have 
\begin{equation*}
\tl{p}(T^*Tx)\leq \ol{p}(T^*)\ol{p}(T)\tl{p}(x)=\ol{p}(T)^2\tl{p}(x)
\end{equation*}
, where we use Lemma \ref{l:autobd} for the equality, and we obtain $\ol{p}(T^*T)\leq p(T)^2$. Finally, for any $p\in S_0(Z)$, $T\in \cL_{b}^{*}(\cE)$ and $x\in\cE$ we have
\begin{align*}
\tl{p}(Tx)^2 &=p([Tx,Tx])=p([T^*Tx,x]) \\
&\leq \tl{p}(T^*Tx)\tl{p}(x) \leq \ol{p}(T^*T)\tl{p}(x)^2
\end{align*}
where for the first inequality we used Lemma \ref{l:schwarzforposop}. Therefore $\ol{p}(T)^2\leq \ol{p}(T^*T)$ and combining the two inequalities we obtain $\ol{p}(T)^2 = \ol{p}(T^*T)$ and the $C^*$ identity is shown for the seminorms $\ol{p}$, completing the proof.
\end{proof}

In the end of this section, we provide some immediate corollaries of the facts proven so far.

\begin{corollary}\label{c:bddadjloc} 
Let $\cE$ be a VH-space over the admissible space $Z$. Then $\cL_{b}^{*}(\cE,S_0(Z))$ and in particular $\cL_{b}^{*}(\cE)$ is a locally $C^*$-algebra.
\end{corollary}

\begin{proof}
Combine the statements of Propositions \ref{p:bddadjcomp} and \ref{p:bddadjpreloc}. 
\end{proof}

\begin{corollary}\label{c:bddadjhilb} 
Let $\cE$ and $\cF$ be two VH-spaces over the admissible space $Z$. 
Then $\cL_{b}^{*}(\cE,\cF,S_0(Z))$ is a locally Hilbert $C^*$-module over the locally $C^*$-algebra 
$\cL_{b}^{*}(\cE,S_0(Z))$. In particular, $\cL_{b}^{*}(\cE,\cF)$ is a locally Hilbert $C^*$-module over the locally $C^*$-algebra 
$\cL_{b}^{*}(\cE)$.
\end{corollary}

\begin{proof}
Given $T \in \cL_{b}^{*}(\cE,\cF,S_0(Z))$ and $A\in \cL_{b}^{*}(\cE,S_0(Z))$, define a right module action of 
$\cL_{b}^{*}(\cE,S_0(Z))$ on $\cL_{b}^{*}(\cE,\cF,S_0(Z))$ by $TA$. Define a pairing 
$[\cdot,\cdot]\colon \cL_{b}^{*}(\cE,\cF,S_0(Z)) \times \cL_{b}^{*}(\cE,\cF,S_0(Z)) \ra \cL_{b}^{*}(\cE,S_0(Z))$ 
by $[T_1, T_2]:= T_1^*T_2$. Then routine checking shows that $[\cdot,\cdot]$ is a 
gramian respecting the right module action. Together with Proposition \ref{p:bddadjcomp} this 
proves the corollary. 
\end{proof}

\section{Adjointable Operators of Barreled Topological VE-Spaces}\label{s:adop}

\subsection{Definition and Examples of Barreled Topological VE-Spaces.}\label{ss:debarvh}

Barreled spaces were introduced by Bourbaki in \cite{Bourbaki}. Recall that, a set in a topological vector space is called a \emph{barrel} if it is closed, absorbent, balanced and convex. A locally convex space is called \emph{barreled} if each barrel is a neighborhood of $0$. In this article, we will use standart results on barreled locally convex spaces, see e.g. Chapter 11 of \cite{Jarchow} or \cite{Narici}. We will refer to these results as needed. 

Let $\cE$ be a topological VE-space over a topologically ordered $*$-space $Z$. If the underlying locally convex topology of $\cE$ is barreled, then $\cE$ is called a \emph{barreled topological VE-space}. If $Z$ is a normed ordered $*$-space and the underlying normed topology of $\cE$ is barreled, then we call $\cE$ a \emph{barreled normed VE-space}. 

Note that any normed VH-space over a Banach ordered $*$-space, as well as any subspace of it is automatically barreled as it is Banach. In particular, any Hilbert $C^*$-module and its closed linear submanifolds are examples. Given a countable family of VH-spaces, by forming their countable direct product VH-space, see Examples 1.4. (4) of \cite{AyGheondea2} for the details, we obtain examples of barreled VH-spaces, as in particular  the underlying locally convex space is a Fr{\'e}chet space, which is well known to be barreled. Also, any barreled subspace of the previous examples are barreled topological VE-spaces. In the following, we provide some non-examples. 

\begin{example}\label{e:nbtopve}
It is a natural question that, given a barreled topologically ordered $*$-space $Z$, if we automatically have that any topological VE-space over $Z$ is also barreled or not. It is not difficult to see that there is no such permanence property, simply by considering any non-barreled subspace of an inner product space. In the following, we provide such an example. 

Consider the complex sequence space of all finite sequences $c_{00}$ with the $l^2$ inner product, so $c_{00}$ is an inner product space and  hence a topological VE-space over $\CC$. By considering the sequence of linear bounded functionals $f_n(x):=\sum_{i=1}^{n}x_i$, $n=1,2,\dots$, we see that the set $(f_n)_{n=1}^{\infty}$ is pointwise bounded, but since $\Vert f_n \Vert = \sqrt{n}$, it is not uniformly bounded. Since the Principle of Uniform Boundedness fails, $c_{00}$ is not barreled. 

Notice that, this example in particular shows that, a topological VE-space need not be barreled. 
\end{example}
Another natural question is, in the case when we have a VH-space $\cH$ if we automatically have barreledness of $\cH$ or not. Notice that, a counter example must be out of the category of normed spaces, since a normed VH-space, being a Banach space in particular, is automatically barreled. In the following, we provide an example of a non-barreled locally $C^*$-algebra, hence an example of a non-barreled locally Hilbert $C^*$-module, and hence an example of a non-barreled VH-space, see Examples \ref{e:evhs}. 

\begin{example}\label{e:nbvhs}
We adopt Example 11.12.6 in \cite{Narici}, and observe that it is a locally $C^*$-algebra. 
For the convenience of the reader, we provide all the details here. Let $T$ be a set of uncountable cardinality and let $X$ be the complex vector space $X\colon = \{ x\in \CC^{T} \mid x(t)=0 \,\, \textrm{except for finitely many t} \}$. Define product of two elements in $X$ pointwise; i.e. $(xy)(t):=x(t)y(t)$, clearly the product is commutative. Similarly, the involution $*$ is defined pointwise; $x^*(t):=\ol{x(t)}$. It is easy to check that $X$ is a $*$-algebra. For each $t\in T$, consider the seminorm $p_{t}(x):=\vert x(t) \vert$ and notice that it is a $C^*$-seminorm. The family of $C^*$-seminorms $\{p_{t}\}_{t\in T}$ turn $X$ into a pre-locally $C^*$-algebra. However, by direct checking, it is not difficult to see that $X$ is a complete locally convex space, therefore $X$ is a locally $C^*$-algebra. 

Now consider the set $B:=\{ x\in X \mid \sum_{t\in T} \vert x(t) \vert \leq 1 \}$. It is straightforward to check that $B$ is a barrel in $X$. We show that $B$ cannot contain a basic neighborhood $U:= \{ x\in X \mid p_{t}(x) \leq r_{t}\,\,\textrm{for all t} \}$ of $0$, where $r_{t} > 0$, hence it cannot be a neighborhood of $0$: Assume $B \supseteq U$ for some $U$. Consider the sets $A_n:=\{ t\in T \mid \vert r_t > 1/n \vert \}$ for $n \geq 2$. Since $T$ is uncountable it is clear that there exists $n_{0}$ such that $A_{n_0}$ contains infinitely many elements. Hence there exists a finite subset $H\subset T$ such that $\sum_{t\in H}r_t > 1$. Now form the element
$x(t) = 
\begin{cases}
r_t, \,\,t\in H \\
0, \,\, t\notin H
\end{cases}
$
which is in $U$, but not in $B$. Therefore $X$ is not barreled. 
\end{example}


\subsection{Automatic Boundedness of Adjointable Operators of Barrelled Topological VE-Spaces.}\label{ss:adop}

In this section, we prove the main result of this article, Theorem \ref{t:adopbar}, which states that, an adjointable operator $T\colon \cE \ra \cF$ from a barreled topological VE-space $\cE$ to a topological VE-space $\cF$ over the same topologically ordered $*$-space $Z$ is automatically bounded, hence continuous. 

 We now do some preparation in order to prove Theorem \ref{t:adopbar}. 
For a topologically ordered $*$-space $Z$, let $p\in S_*(Z)$ and consider the kernel of $p$, $I_p:=\{ z\in Z \mid p(z)=0 \}$. It is easy to see that $I_p$ is a closed ordered $*$-subspace of $Z$. In addition, since $p$ is an increasing $*$-seminorm it follows that $I_p$ is an \emph{order ideal}, that is, $I_p$ is selfadjoint and  if $z_0\in I_p$ with $z_0 \geq 0$, for any $z\in Z$ with $0\geq z \geq z_0$ we have $z\in I_p$. 

It follows from Example 1.2(9) of \cite{AyGheondea1} that, the quotient space $Z_p:=Z/I_p$ is an ordered $*$-space with the cone $Z_p^{+}:=Z^{+}/I_p$, and the involution of $Z_p$ is defined by $(z+I_p)^*:=z^*+I_p$. 

It is standard that, $\Vert z+I_p \Vert_p := p(z)$ defines a norm on $Z_p$ for any $p\in S_*(Z)$. It is easy to see that $\Vert \cdot \Vert_p$ is a $*$-norm, as $p$ is a $*$-seminorm, and to see that it is increasing, we just observe that for $y,z \in Z^+$ we have $\Vert z+I_p \Vert_p =p(z) \leq p(z+y) = \Vert (z+y)+I_p \Vert_p = \Vert (z+I_p) + (y+I_p) \Vert_p$. 

Therefore the ordered $*$-space $(Z_p,Z_p^{+})$ with the norm $\Vert \cdot \Vert_p$ satisfies all axioms (a3) -- (a6) of a topologically ordered $*$-space except for (a5). By Lemma \ref{l:clcone} it follows that for any $p\in S_*(Z)$ the space $(Z_p,\ol{Z_p^{+}})$ where $\ol{Z_p^{+}}$ is the closure of $Z_p^{+}$ is a normed ordered $*$-space with the norm $\Vert \cdot \Vert_p$ an increasing $*$-norm.  

We record the discussions above as a lemma:

\begin{lemma}\label{l:tosquot}
Let $Z$ be a topologically ordered $*$-space and $p\in S_*(Z)$. Then the quotient space $(Z_p,Z_p^{+})$ is an ordered $*$-space. Moreover, the space $(Z_p,Z_p^{+},\Vert \cdot \Vert_p)$ satisfies all axioms of a normed ordered $*$-space except for closedness of its cone and by passing to the closure $\ol{Z_p^{+}}$ of the cone, the space $(Z_p,\ol{Z_p^{+}},\Vert \cdot \Vert_p)$ is a normed ordered $*$-space.
\end{lemma}

\begin{proof}
See the preceeding discussion.
\end{proof}

From now on we will use $Z_p$ for the space $(Z_p,\ol{Z_p^{+}},\Vert \cdot \Vert_p)$ as in Lemma \ref{l:tosquot}, where $p\in S_*(Z)$. 

Now let $\cE$ be a topological VE-space over the topologically ordered $*$-space $Z$ and $p\in S_*(Z)$. Following the notation in \cite{AyGheondea2}, define $\tl{I^{\cE}_p}:= \{ x\in E \mid [x,x]\in I_p\}$. Equivalently, $\tl{I^{\cE}_p} = \{ x\in E \mid \tl{p}(x)=0 \}$, where $\tl{p}$ is as in \eqref{e:qujeh}. We will use the notation $\tl{I_p}$ as well if the space $\cE$ is clear from the context. 

\begin{lemma}\label{l:vequot}
Let $\cE$ be a topological VE-space over the topologically ordered $*$-space $Z$ and let $p\in S_*(Z)$. Then $\tl{I_p}$ is a closed vector subspace of $\cE$ and the quotient $\cE_p:=\cE/\tl{I_p}$ is a normed VE-space over the normed ordered $*$-space $Z_p$, with norm given by $\Vert x+\tl{I_p} \Vert=\tl{p}(x)$ for any $x\in\cE$.
\end{lemma}

\begin{proof}
Let $p\in S_*(Z)$. For any $x\in\tl{I_p}$ and $\alpha\in\CC$, it is immediate to see that $\alpha x \in \tl{I_p}$. Now letting $x_1, x_2 \in \tl{I_p}$ we have
\begin{align*}
\tl{p}^2(x_1+x_2) &= p([x_1+x_2, x_1+x_2]) \\
&\leq p([x_1,x_1]) + 2p([x_1,x_2]) + p([x_2,x_2]) \\
&\leq \tl{p}^2(x_1) + 8\tl{p}(x_1)\tl{p}(x_2) + \tl{p}^2(x_2) = 0
\end{align*}
where for the second inequality we used the Schwarz type inequality in Lemma \ref{l:schwarz}. This completes the proof that $\tl{I_p}$ is a vector subspace of $\cE$. It is a routine check to see that $\tl{I_p}$ is closed. 

We form the quotient normed space $\cE_p:=\cE/\tl{I_p}$ with quotient norm $\Vert x+\tl{I_p} \Vert_p = \tl{p}(x)$ for any $x\in \cE$ and show that it is a normed VE-space over $Z_p$. Define the pairing $[\cdot,\cdot]_p:\cE_p \times \cE_p \ra Z_p$ by 
\begin{equation*}
[x+\tl{I_p}, y+\tl{I_p}]_p:= [x,y] + I_p
\end{equation*}
for any $x,y \in \cE$. To see that this is well defined, let $x_1, x_2, y_1, y_2 \in \cE$ such that $x_1+\tl{I_p} = x_2+\tl{I_p}$ and $y_1+\tl{I_p} = y_2+\tl{I_p}$, equivalently, $\tl{p}(x_1-x_2) = 0$ and $\tl{p}(y_1-y_2) = 0$. Observe that we have 
\begin{align*}
p([x_1,y_1] - [x_2,y_2]) &= p([x_1,y_1]+[x_1,y_2]-[x_1,y_2]-[x_2,y_2]) \\
&\leq p([x_1,y_1-y_2]) + p([x_1-x_2, y_2]) \\
&\leq 4\tl{p}(x_1)\tl{p}(y_1-y_2) + 4\tl{p}(x_1-x_2)\tl{p}(y_2) = 0 
\end{align*}
where for the second inequality we used Lemma \ref{l:schwarz} again. Then we have
\begin{align*}
[x_1+\tl{I_p}, y_1+\tl{I_p}]_p &= [x_1+y_1]+I_p \\
&=[x_2,y_2]+I_p \\
&=[x_2+\tl{I_p}, y_2+\tl{I_p}]_p
\end{align*}
so the pairing $[\cdot,\cdot]_p$ is well defined. It is straightforward to check that $[\cdot,\cdot]_p$ defines a gramian on $\cE_p$. Finally, the natural normed topology of $\cE_p$ as in Subsection \ref{ss:vhs}, given by
\begin{align*}
\Vert x+ \tl{I_p}\Vert &:= \Vert [x+\tl{I_p},x+\tl{I_p}]_p \Vert_p^{1/2}\\
&=\Vert [x,x]+I_p] \Vert^{\frac{1}{2}}=p([x,x])^{1/2}=\tl{p}(x)
\end{align*}
for any $x\in \cE$, agrees with the quotient norm $\Vert x+\tl{I_p} \Vert_p = \tl{p}(x)$ and therefore the quotient norm turns $\cE_p$ into a normed VE-space. That completes the proof.
\end{proof}

\begin{corollary}\label{c:barvequot}
Let $\cE$ be a barreled topological VE-space over the topologically ordered $*$-space $Z$ and let $p\in S_*(Z)$. Then $\cE_p$ is a barreled normed VE-space over $Z_p$. 
\end{corollary}

\begin{proof}
Since $\tl{I_p}$ is a closed subspace of $\cE$ and $\cE_p$ has the quotient topology by Lemma \ref{l:vequot}, by e.g. Proposition 1(a) in 11.3 of \cite{Jarchow}, the normed topology of $\cE_p$ is barreled as well. Therefore $\cE_p$ is a barreled normed VE-space over $Z_p$.
\end{proof}

\begin{corollary}\label{c:oprquot}
Let $\cE$ and $\cF$ be two topological VE-spaces over the same topologically ordered $*$-space $Z$. Let $p\in S_*(Z)$ and $T\in \cL^*(\cE,\cF)$. Then the naturally defined quotient operator $T_p\colon \cE_p \ra \cF_p$ is a well defined linear operator and $T_p\in \cL^*(\cE_p,\cF_p)$. 
\end{corollary}

\begin{proof}
Define $T_p\colon \cE_p \ra \cF_p$ by $T_p(x+\tl{I^{\cE}_{p}}):= Tx + \tl{I^{\cF}_{p}}$. In order to see that $T_p$ is well defined, it is enough to show that, if $x+\tl{I^{\cE}_p}=0$, then $Tx + \tl{I^{\cF}_p} = 0$; equivalently, $\tl{p}_{\cE}(x) = 0$ implies $\tl{p}_{\cF}(Tx)=0$. We have, for such $x\in \cE$, 
\begin{align*}
\tl{p}_{\cF}(Tx) &= p([Tx,Tx])^{\frac{1}{2}} = p([T^*Tx,x])^{\frac{1}{2}} \\
&\leq \tl{p}_{\cE}(T^*Tx)\tl{p}_{\cE}(x) = 0
\end{align*}
where for the inequality we used Lemma \ref{l:schwarzforposop}. 

It is easy to see that $T_p$ is linear. To see that it is adjointable with adjoint $(T^*)_p$, let $x\in \cE$ and $y\in \cF$ be arbitrary. Then
\begin{align*}
&[T_p(x+\tl{I^{\cE}_p}), y+\tl{I^{\cF}_p}]_{\cF_p} = [Tx+\tl{I^{\cF}_p}, y+\tl{I^{\cF}_p}]_{\cF_p} \\
=& [Tx,y]+I_p = [x,T^*y]+I_p \\
=&[x+\tl{I^{\cE}_p}, T^*y+\tl{I^{\cE}_p}]_{\cE_p} = [x+\tl{I^{\cE}_p}, (T^*)_p(y+\tl{I^{\cF}_p})]_{\cE_p}
\end{align*}
and it is shown that $T_p$ is adjointable with $T_p^* = (T^*)_p$. 
\end{proof}

The following theorem establishes boundedness of an adjointable operator on a barreled topological VE-space. To some extent, we use an idea going back to \cite{Paschke}, see the argument in page 8 of \cite{Lance} as well. 

\begin{theorem}\label{t:adopbar}
Let $\cE$ be a barreled topological VE-space and $\cF$ be a topological VE-space over the same 
topologically ordered $*$-space $Z$. Then $\cL^{*}(\cE,\cF)=\cL_{b}^{*}(\cE,\cF)$. 
\end{theorem}

\begin{proof}
Let $T\in\cL^{*}(\cE,\cF)$ be an operator, and $p\in S_*(Z)$. Using Lemma \ref{l:vequot} and Corollaries \ref{c:barvequot} and \ref{c:oprquot} we obtain the barreled normed VE-space $\cE_p$, the normed VE-space $\cF_p$, both over $Z_p$ and the quotient operator $T_p\colon \cE_p \ra \cF_p$. Now for any $y\in \cF$ with $\tl{p}(y)\leq 1$ define the vector valued linear functional $f_y\colon \cE_p \ra \cF_p$ by $f_y(x+\tl{I}_p) = [T^*_p(y+\tl{I}_p), x+\tl{I}_p]_{\cE_p}$. 

We show that for $y$ fixed, $f_y$ is bounded. We have 
\begin{align*}
\Vert f_y(x+\tl{I}_p) \Vert_{Z_p} &= \Vert [T^*_p(y+\tl{I}_p), x+\tl{I}_p]_{\cE_p} \Vert_{Z_p} \\
&=\Vert [T^*y, x]_{\cE} + I_p \Vert_{Z_p} \\
&= p([T^*y,x]) \\
&\leq 4\tl{p}(T^*y)\tl{p}(x) = 4\tl{p}(T^*y)\Vert (x+\tl{I}_p) \Vert_{\cE_p}
\end{align*}
for all $x\in\cE$, where for the inequality we used Lemma \ref{l:schwarz} and it follows that $f_y$ is bounded with $\Vert f_y \Vert\leq 4\tl{p}(T^*y)$. In addition we show that the family $\{f_y\}$ is pointwise bounded. For any fixed $x\in\cE$ we have
\begin{align*}
\Vert f_y(x+\tl{I}_p) \Vert_{Z_p} &= \Vert [T^*_p(y+\tl{I}_p), x+\tl{I}_p]_{\cE_p} \Vert_{Z_p} \\
&=\Vert [T^*y, x]_{\cE} + I_p \Vert_{Z_p} \\
&= p([T^*y,x]) = p([y,Tx])\\
&\leq 4\tl{p}(y)\tl{p}(Tx) \leq 4\tl{p}(Tx)
\end{align*}
where we used Lemma \ref{l:schwarz} again for the first inequality. Hence the family $\{f_y\}$ is pointwise bounded with constant $C_x=4\tl{p}(Tx)$ for any fixed $x\in\cE$. 

It follows by the Principle of Uniform Boundedness 
applied to the family $\{f_y\}$ that, there exists a constant $C >0$ such that $\Vert f_y(x+\tl{I}_p) \Vert_{Z_p} \leq C \Vert x+\tl{I}_p \Vert_{\cE_p}$ for all $x\in\cE$ and $y\in\cF$ with $\tl{p}(y)\leq 1$. Equivalently, we obtain $p([y, Tx]) \leq C \tl{p}(x) = Cp([x,x])^{1/2}$ for all $x\in\cE$ and $y\in\cF$ with $\tl{p}(y)\leq 1$.

Now let $y\in\cF$ be any element, and let $\epsilon > 0$. By applying the inequality above with the element $(\tl{p}(y) + \epsilon)^{-1} y$ in place of $y$ we obtain 
\begin{equation*}
p([y,Tx]) \leq C (\tl{p}(y) + \epsilon) \tl{p}(x)
\end{equation*}
for all $x\in\cE$ and $y\in\cF$. Letting $\epsilon \ra 0$, it follows that 
\begin{equation*}
p([y,Tx]) \leq C \tl{p}(y) \tl{p}(x)
\end{equation*}
for all $x\in\cE$ and $y\in\cF$. Finally, choosing $y:= Tx$ in the last inequality and a standard argument implies that $\tl{p}(Tx) \leq C \tl{p}(x)$ for all $x\in\cE$. It follows that $T$ is bounded and the proof is finished.
\end{proof}

Notice that, in the preceding proof, if the topological VE-space $\cE$ is 
such that the quotient normed spaces $\cE_p$ are barreled, then the proof still works 
even without the barreled assumption on $\cE$. 
Consequently, we obtain the following corollary:

\begin{corollary}\label{c:quotbar}
Let $\cE$ be a topological VE-space over the topologically ordered $*$-space $Z$. Assume that the quotient normed VE-spaces $\cE_p$ 
are barreled for all $p\in S_0(Z)$, where $S_0(Z)\subseteq S_*(Z)$ is a family defining 
the topology of $Z$. Let $\cF$ be a topological VE-space over the same 
topologically ordered $*$-space $Z$. Then $\cL^{*}(\cE,\cF)=\cL_{b}^{*}(\cE,\cF,S_0(Z))$.
\end{corollary}


\begin{remark}\label{r:opcleq}
As a consequence of Theorem \ref{t:adopbar}, with $\cE$ and $\cF$ as in the theorem, we have $\cL_{b}^{*}(\cE,\cF) = \cL^*_{\mathrm{c}}(\cE,\cF) = \cL^*(\cE,\cF)$. In particular, any adjointable operator is bounded and hence continuous, with a bounded and hence continuous adjoint. 
\end{remark}

Let $\cH$ be a VH-space over a normed admissible space $Z$. Then $\cH$ is in particular a Banach space, see subsection \ref{ss:vhs}, consequently it is a barreled VH-space. Hence the following corollary of Theorem \ref{t:adopbar} follows:

\begin{corollary}
Let $\cH$ be a normed VH-space and $\cK$ be a normed VE-space over the same normed admissible space $Z$. Then $\cL_{b}^{*}(\cH,\cK)=\cL^{*}(\cH,\cK)$.
\end{corollary}

We obtain the following structural theorem for $\cL^{*}(\cE)$ when $\cE$ is a barreled topological VE-space.

\begin{theorem}\label{t:adopbarloc}
Let $\cE$ be a barreled topological VE-space over a topologically ordered $*$-space $Z$. Then $\cL^{*}(\cE)$ is a pre-locally $C^*$-algebra. In addition, if $\cE$ is a barreled VH-space over an admissible space $Z$, then $\cL^{*}(\cE)$ is a locally $C^*$-algebra. 
\end{theorem}

\begin{proof}
For the first statement, combine the statements of Proposition \ref{p:bddadjpreloc} and Theorem \ref{t:adopbar}. For the second, combine the statements of Corollary \ref{c:bddadjloc} and Theorem \ref{t:adopbar}.
\end{proof}

\begin{corollary}\label{c:adopbarhilb}
Let $\cH$ be a barreled VH-space over the admissible space 
$Z$ and $\cK$ be a VH-space over $Z$. Then the space $\cL^*(\cH,\cK)$ is 
a locally Hilbert $C^*$-module over the locally $C^*$-algebra $\cL^*(\cH)$. 
\end{corollary}

\begin{proof}
Combine the statements of Corollary \ref{c:bddadjhilb}, Theorem \ref{t:adopbar} and 
Corollary \ref{t:adopbarloc}. 
\end{proof}

We obtain the following corollary from Theorem \ref{t:adopbarloc}. 

\begin{corollary}
Let $\cH$ be a normed VH-space over the normed admissible space $Z$. Then $\cL^{*}(\cH)$ is a $C^*$-algebra.
\end{corollary}

The hypothesis of Corollary \ref{c:quotbar} holds when $\cE$ is a 
locally Hilbert $C^*$-module over a locally $C^*$-algebra $\cA$ and the family $S_0(Z)$ is taken as
a generating family of $C^*$-seminorms $S^*(\cA)$ 
of $\cA$. It is known in that case, see \cite{Apostol}, that the quotient spaces $\cA_p$, $p\in S^*(\cA)$ are $C^*$-algebras and corresponding quotient spaces $\cE_p$ are Hilbert $C^*$-modules. In particular $\cE_p$ are complete and hence barreled, e.g. see \cite{Phillips} and \cite{Zhuraev}. 

As a result, we recover the following well known result, see \cite{Phillips} and 
\cite{Zhuraev}, as a special case of Corollary \ref{c:quotbar}:

\begin{corollary}\label{c:quothilb}
Let $\cE$ be a locally Hilbert $C^*$-module over the locally $C^*$-algebra $\cA$ 
and $\cF$ be a pre-locally Hilbert $C^*$-module over $\cA$. Then $\cL^{*}(\cE,\cF)=\cL_{b}^{*}(\cE,\cF,S^*(\cA))$. Moreover,
$\cL^*(\cE)=\cL_{b}^{*}(\cE,S^*(\cA))$ is a locally $C^*$-algebra. If, in addition $\cF$ is also a locally Hilbert $C^*$-module, then $\cL^{*}(\cE,\cF)=\cL_{b}^{*}(\cE,\cF,S^*(\cA))$ is a locally Hilbert $C^*$-module over the locally $C^*$-algebra $\cL^*(\cE,S^*(\cA))$. 
\end{corollary}

\begin{proof}
In order to prove the first statement, note that $\cE$ is a VH-space in particular and $\cF$ is a topologically VE-space over the admissible space $\cA$, see Examples \ref{e:evhs}, then use Corollary \ref{c:quotbar} with $S_0(\cA)=S^*(\cA)$. For the second statement, combine Corollary \ref{c:bddadjloc} and the first statement. Finally, for the third statement, combine Corollary \ref{c:bddadjhilb} and first statement. 
\end{proof}


\section{Dilation Theory of Barreled VH-Spaces}

In this section we pick an implication of Theorem \ref{t:adopbar} 
on dilations of barreled VH-spaces, as an application of 
Theorem \ref{t:adopbar}. More precisely, we will show that when barreled VH-spaces are considered, an operator boundedness condition from a characterization of the existence of invariant VH-space linearizations (equivalently, existence of reproducing kernel VH-spaces with $*$-representations) of kernels valued in operators of VH-spaces, see Theorem \ref{t:vhinvkolmo2}, is satisfied automatically. 

First we recall some definitions from \cite{AyGheondea2} that will be necessary in the remainder of the paper, mainly for the convenience of the reader. For detailed discussions of these definitons and related results we refer to \cite{AyGheondea2}.

Let $X$ be a nonempty set and let $\cE$ be a VE-space over the ordered 
$*$-space $Z$. 
A map $\fk\colon X\times X\ra \cL(\cE)$ is called a \emph{kernel} on $X$ and 
valued in 
$\cL(\cE)$. In case the kernel $\fk$ has all its values in $\cL^*(\cE)$, 
an \emph{adjoint} kernel 
$\fk^*\colon X\times X\ra\cL^*(\cE)$ can be associated by 
$\fk^*(x,y)=\fk(y,x)^*$ for all 
$x,y\in X$. The kernel $\fk$ is called \emph{Hermitian} if $\fk^*=\fk$. In what follows we will always consider kernels valued in $\cL^*(\cE)$. 

Given $n\in\NN$, the kernel $k$ is called \emph{$n$-positive} if for any 
$x_1,x_2,\ldots,x_n\in X$ and any $h_1,h_2,\ldots,h_n\in\cH$ we have
\begin{equation}\label{e:npos} \sum_{i,j=1}^n [\fk(x_i,x_j)h_j,h_i]_\cE\geq 0.
\end{equation}
 The kernel $k$ is called \emph{positive semidefinite} 
(or \emph{of positive type}) if it is $n$-positive for all natural numbers $n$. A 2-positive kernel is Hermitian, see \cite{Gheondea}.

Given an $\cL^*(\cE)$-valued kernel $\fk$ on a nonempty set $X$, for some 
VE-space $\cE$ on an ordered $*$-space $Z$,  a \emph{VE-space 
linearisation} or, equivalently, a
\emph{VE-space Kolmogorov decomposition} of $\fk$ is, by definition, 
a pair $(\cK;V)$, subject to the following conditions:
  
  \begin{itemize}
  \item[(vel1)] $\cK$ is a VE-space over the same ordered $*$-space $Z$.
  \item[(vel2)] $V\colon X\ra\cL^*(\cE,\cK)$ satisfies $\fk(x,y)=V(x)^*V(y)$ 
for all $x,y\in X$. \end{itemize}
The VE-space linearisation $(\cK;V)$ is called \emph{minimal} if
  \begin{itemize}
  \item[(vel3)] $\lin V(X)\cE=\cK$.
  \end{itemize}
Two VE-space linearisations $(V;\cK)$ and $(V';\cK')$ 
of the same kernel $\fk$ are 
called \emph{unitary equivalent} if there exists a VE-space isomorphism 
$U\colon \cK\ra\cK'$ such that $UV(x)=V'(x)$ for all $x\in X$.
  
A minimal VE-space 
linearisation $(\cK;V)$ of a positive semidefinite kernel $\fk$, 
is unique modulo unitary equivalence, see \cite{AyGheondea1}.

Let $\cE$ be a VE-space over the ordered $*$-space $Z$, and let $X$ be a 
nonempty set. Let $\cF=\cF(X;\cE)$ denote the complex vector space of all functions 
$f\colon X\ra \cE$. A VE-space $\cR$, 
over the same ordered $*$-space 
$Z$, is called an \emph{$\cH$-reproducing kernel VE-space on $X$} 
if there exists a Hermitian kernel $\fk\colon X\times X\ra\cL^*(\cE)$ 
such that the following axioms are satisfied:
\begin{itemize} 
\item[(rk1)] $\cR$ is a subspace of $\cF(X;\cE)$, with all algebraic operations.
\item[(rk2)] For all $x\in X$ and all $h\in\cE$, 
the $\cH$-valued function $\fk_x h=\fk(\cdot,x)h\in\cR$.
\item[(rk3)] For all $f\in\cR$ we have $[f(x),h]_\cE=[f,k_x h]_\cR$, for all 
$x\in X$ and $h\in \cE$.
\end{itemize}
As a consequence of (rk2), $\lin\{\fk_xh\mid x\in X,\ h\in\cE\}\subseteq \cR$.
The reproducing kernel VE-space $\cR$ is called \emph{minimal} if
the following property holds as well:
\begin{itemize}
\item[(rk4)] $\lin\{\fk_xh\mid x\in X,\ h\in \cE\}=\cR$.
\end{itemize}

Let $\Gamma$ be a (multiplicative) $*$-semigroup, that is, there is an 
\emph{involution} $*$ on $\Gamma$: $(\xi\eta)^*=\eta^* \xi^*$ 
and $(\xi^*)^*=\xi$ for all $\xi,\eta\in\Gamma$. Note that, in case $\Gamma$ has a unit $\epsilon$ then 
$\epsilon^*=\epsilon$. Let $\Gamma$ act on a nonempty set $X$,
denoted by $\xi\cdot x$, for all $\xi\in\Gamma$ 
and all $x\in X$. By definition, we have 
\begin{equation}\label{e:action}
\alpha\cdot(\beta\cdot x)=(\alpha\beta)\cdot x\mbox{ for all }\alpha,
\beta\in \Gamma\mbox{ and all }x\in X.\end{equation} 

Given a VE-space $\cE$ we consider those Hermitian kernels 
$\fk\colon X\times X\ra\cL^*(\cE)$ that are \emph{invariant} under the action 
of $\Gamma$ on $X$, that is,
\begin{equation}\label{e:invariant} \fk(y,\xi\cdot x)=\fk(\xi^*\cdot y,x)
\mbox{ for all }x,y\in X\mbox{ and all }\xi\in\Gamma.
\end{equation}
A triple $(\cK;\pi;V)$ is called an \emph{invariant VE-space linearisation} 
of  the kernel $\fk$ and the action of $\Gamma$ on $X$, shortly a
\emph{$\Gamma$-invariant VE-space linearisation} of $\fk$, if:
\begin{itemize}
\item[(ikd1)] $(\cK;V)$ is a VE-space linearisation of the kernel $\fk$.
\item[(ikd2)] $\pi\colon \Gamma\ra\cL^*(\cK)$ is a $*$-representation, that
  is, a multiplicative $*$-morphism.
\item[(ikd3)] $V$ and $\pi$ are related by the formula: 
$V(\xi\cdot x)= \pi(\xi)V(x)$, for all $x\in X$, $\xi\in\Gamma$.
\end{itemize}

If $(\cK;\pi;V)$ is a $\Gamma$-invariant VE-space 
linearisation of the kernel 
$\fk$ then $\fk$ is invariant under the action of $\Gamma$ on $X$.

If, in addition to the axioms (ikd1)--(ikd3), the triple $(\cK;\pi;V)$ 
has the property
\begin{itemize}
\item[(ikd4)] $\lin V(X)\cE=\cK$,
\end{itemize} that is, the VE-space linearisation $(\cK;V)$ is minimal,
then $(\cK;\pi;V)$ is called a \emph{minimal
$\Gamma$-invariant VE-space linearisation} of $\fk$ and the action of 
$\Gamma$ on $X$.

Now let $\cH$ be a VH-space over the admissible space $Z$, and 
consider a kernel $\fk\colon X\times X\ra \cL_c^*(\cH)$. A \emph{VH-space 
linearisation} of $\fk$, or  
\emph{VH-space Kolmogorov decomposition} of $\fk$, is
a pair $(\cK;V)$, subject to the following conditions:
  
  \begin{itemize}
  \item[(vhl1)] $\cK$ is a VH-space over the same ordered $*$-space $Z$.
  \item[(vhl2)] $V\colon X\ra\cL^*_c(\cH,\cK)$ satisfies $\fk(x,y)=V(x)^*V(y)$ 
for all $x,y\in X$. \end{itemize}
The VH-space linearisation $(\cK;V)$ is called \emph{minimal} if
  \begin{itemize}
  \item[(vhl3)] $\lin V(X)\cH$ is dense in $\cK$.
  \end{itemize}

Two VH-space linearisations $(V;\cK)$ and $(V';\cK')$ 
of the same kernel $\fk$ are 
called \emph{unitary equivalent} if there exists a unitary operator 
$U\colon \cK\ra\cK'$ such that $UV(x)=V'(x)$ for all $x\in X$.
  
As in the case for VE-space linearisations, a minimal VH-space 
linearisation $(\cK;V)$ of a positive semidefinite kernel $\fk$, is unique modulo unitary equivalence, 
see \cite{Gheondea}.

A VH-space $\cR$ over the ordered $*$-space 
$Z$ is called an \emph{$\cH$-reproducing kernel VH-space on $X$} 
if there exists a Hermitian kernel $\fk\colon X\times X\ra\cL^*_c(\cH)$ 
such that the following axioms are satisfied:
\begin{itemize} 
\item[(rk1)] $\cR$ is a subspace of $\cF(X;\cH)$, with all algebraic operations.
\item[(rk2)] For all $x\in X$ and all $h\in\cH$, 
the $\cH$-valued function $\fk_x h=\fk(\cdot,x)h\in\cR$.
\item[(rk3)] For all $f\in\cR$ we have $[f(x),h]_\cH=[f,\fk_x h]_\cR$, for all 
$x\in X$ and $h\in \cH$.
\item[(rk4)] For all $x\in X$ the evaluation operator $\cR\ni f\mapsto f(x)\in
\cH$ is continuous.
\end{itemize}

We now describe invariant linearisations for VH-spaces. Let $\cH$ be a VH-space over an admissible space $Z$, let $\fk\colon X\times
X\ra\cL_c^*(\cH)$ be a kernel on some nonempty set $X$, and let $\Gamma$ be a
$*$-semigroup that acts at left on $X$. As in the case of VE-space operator
valued kernels, we call $\fk$ \emph{$\Gamma$-invariant} if
\begin{equation}\fk(\xi\cdot x,y)=\fk(x,\xi^*\cdot y),\quad
  \xi\in\Gamma,\ x,y\in X.
\end{equation} 
A triple $(\cK;\pi;V)$ is called
a \emph{$\Gamma$-invariant VH-space linearisation} for $\fk$ if
\begin{itemize}
\item[(ihl1)] $(\cK;V)$ is a VH-space linearisation of $\fk$.
\item[(ihl2)] $\pi\colon \Gamma\ra\cL^*_c(\cK)$ is a $*$-representation.
\item[(ihl3)] $V(\xi\cdot x)=\pi(\xi)V(x)$ for all $\xi\in\Gamma$ and all $x\in X$.
\end{itemize}
Also, $(\cK;\pi;V)$ is \emph{minimal} if the VH-space linearisation $(\cK;V)$ is
minimal, that is, $\cK$ is the closure of the linear span of $V(X)\cH$.

The following theorem is Theorem 2.10 from \cite{AyGheondea2}, see also the Remark 2.11 following it. It characterizes the existence of invariant VH-space linearisations, equivalently, existence of reproducing kernel VH-spaces (with $*$-representations of the underlying $*$-semigroup) of kernels valued in the class $\cL_{c}^*(\cH)$ of a VH-space $\cH$. 

\begin{theorem} \label{t:vhinvkolmo2} 
Let $\Gamma$ be a $*$-semigroup that acts 
on the nonempty set $X$ and let $\fk\colon X\times X\ra\cL_{c}^*(\cH)$ 
be a kernel, 
for some VH-space $\cH$ over an admissible space $Z$. Let $S_0(Z)\subseteq S(Z)$ be a family of seminorms generating the topology of $Z$.
Then the following assertions are equivalent:

\begin{itemize}
\item[(1)] $\fk$ is positive semidefinite, in the sense of \eqref{e:npos}, 
and invariant under the action of $\Gamma$ on $X$, that is, 
\eqref{e:invariant} holds, and, in addition, the following conditions hold:
\begin{itemize}
\item[(b1)] For any $\xi \in \Gamma$ and 
any seminorm $p \in S_0(Z)$, there exists a seminorm 
$q \in S(Z)$ and 
a constant $c_p(\xi)\geq 0$ such that for all $n \in \NN$,
$\{h_i \}_{i=1}^{n} \in \cH$, $\{ x_i \}_{i=1}^{n} \in X$ we have
\begin{equation*}
p ( \sum_{i,j=1}^{n} [ \fk(\xi\cdot x_i, \xi\cdot x_j)h_j, h_i ]_{\cH} )
\leq c_p(\xi)\,  q( \sum_{i,j=1}^{n} 
[ \fk(x_i, x_j)h_j, h_i ]_{\cH} ). 
\end{equation*} 
\item[(b2)] For any $x \in X$ and any seminorm $p \in S_0(Z)$, 
 there exists 
a seminorm 
$q \in S(Z)$ and 
a constant $c_p(x)\geq 0$ such that for all 
$n \in \NN$, $\{ y_i \}_{i=1}^{n} \in X$, 
$\{ h_i \}_{i=1}^{n} \in \cH$ we have
\begin{equation*}
p ( \sum_{i,j=1}^{n} [ \fk(x,y_i)h_i, \fk(x, y_j)h_j ]_{\cH} )
\leq c_p(x)\,  q ( \sum_{i,j=1}^{n} 
[ \fk(y_j, y_i)h_i, h_j ]_{\cH} ). 
\end{equation*}
\end{itemize}
\item[(2)] $\fk$ has a $\Gamma$-invariant VH-space
linearisation $(\cK;\pi;V)$.
\item[(3)] $\fk$ admits an $\cH$-reproducing kernel VH-space $\cR$ and 
there exists a $*$-representation $\rho\colon \Gamma\ra\cL_{c}^*(\cR)$ such that 
$\rho(\xi)\fk_xh=\fk_{\xi\cdot x}h$ for all $\xi\in\Gamma$, $x\in X$, $h\in\cH$.
\end{itemize}

In addition, in case any of the assertions \emph{(1)}, \emph{(2)}, or
\emph{(3)} holds,  
then a minimal\, $\Gamma$-invariant VH-space linearisation of $\fk$ can be 
constructed, any minimal $\Gamma$-invariant VH-space linearisation of $\fk$ 
is unique up to unitary equivalence, 
and the pair $(\cR;\rho)$ as in assertion \emph{(3)} is uniquely determined by 
$\fk$ as well.
\end{theorem} 

In the final part of the paper, we show that, when barreled VH-spaces are considered, condition (b2) of Theorem \ref{t:vhinvkolmo2} is automatically satisfied. Precisely, we have the following theorem:

\begin{theorem} \label{t:barvhinvkolmo} 
Let $\Gamma$ be a $*$-semigroup that acts 
on the nonempty set $X$ and let $\fk\colon X\times X\ra\cL^*(\cH)$ 
be a kernel, 
for a barreled VH-space $\cH$ over an admissible space $Z$. Let $S_0(Z)\subseteq S_*(Z)$ be a family of seminorms generating the topology of $Z$. 
Then the following assertions are equivalent:

\begin{itemize}
\item[(1)] $\fk$ is positive semidefinite, in the sense of \eqref{e:npos}, 
and invariant under the action of $\Gamma$ on $X$, that is, 
\eqref{e:invariant} holds, and, in addition, the following condition hold:
\begin{itemize}
\item[(b1)] For any $\xi \in \Gamma$ and 
any seminorm $p \in S_0(Z)$, there exists a seminorm 
$q \in S(Z)$ and 
a constant $c_p(\xi)\geq 0$ such that for all $n \in \NN$,
$\{h_i \}_{i=1}^{n} \in \cH$, $\{ x_i \}_{i=1}^{n} \in X$ we have
\begin{equation*}
p ( \sum_{i,j=1}^{n} [ \fk(\xi\cdot x_i, \xi\cdot x_j)h_j, h_i ]_{\cH} )
\leq c_p(\xi)\,  q( \sum_{i,j=1}^{n} 
[ \fk(x_i, x_j)h_j, h_i ]_{\cH} ). 
\end{equation*} 
\end{itemize}
\item[(2)] $\fk$ has a $\Gamma$-invariant VH-space
linearisation $(\cK;\pi;V)$.
\item[(3)] $\fk$ admits an $\cH$-reproducing kernel VH-space $\cR$ and 
there exists a $*$-representation $\rho\colon \Gamma\ra\cL_{c}^*(\cR)$ such that 
$\rho(\xi)\fk_xh=\fk_{\xi\cdot x}h$ for all $\xi\in\Gamma$, $x\in X$, $h\in\cH$.
\end{itemize}

In addition, in case any of the assertions \emph{(1)}, \emph{(2)}, or
\emph{(3)} holds,  
then a minimal\, $\Gamma$-invariant VH-space linearisation of $\fk$ can be 
constructed, any minimal $\Gamma$-invariant VH-space linearisation of $\fk$ 
is unique up to unitary equivalence, 
and the pair $(\cR;\rho)$ as in assertion \emph{(3)} is uniquely determined by 
$\fk$ as well.
\end{theorem} 

Before the proof of this theorem, let us recall the definition of an m-topologisable operator: For a topological VE-space $\cE$ over the topologically 
ordered $*$-space $Z$, following \cite{Zelazko}, also see \cite{Bonet}, an operator $T\in \cL(\cE)$ 
is called \emph{m-topologisable} if for every $p\in S_0(Z)$, where $S_0(Z)\subseteq S(Z)$ is an arbitrary generating family of seminorms, there exists 
a constant $D_p\geq 0$ and a continuous seminorm $r$ on $\cE$ 
such that, for every $n \in \NN$ and every $h\in \cE$,
\begin{equation}\label{e:mtop}
\tl{p}(T^n h)=p([T^nh,T^nh])^{\frac{1}{2}} \leq D_p^n\, r(h).
\end{equation}
It is easy to see that the definition does not depend on the particular family of seminorms $S_0(Z)$ chosen. 

Let us observe that, if $T\in\cL_b(\cE)$, then for every $n \in \NN$, $p\in S_*(Z)$ and $h\in \cE$ we have 
\begin{align*}
\tl{p}(T^n h)&=p([T^nh,T^nh])^{\frac{1}{2}} \leq C_p\, p([T^{n-1}h,T^{n-1}h])^{\frac{1}{2}} \\
&\leq C_p^2\, p([T^{n-2}h,T^{n-2}h])^{\frac{1}{2}} \leq \dots \leq C_p^n\, p(h,h])^{\frac{1}{2}} = C_p^n\, \tl{p}(h)
\end{align*} 
for some constant $C_p \geq 0$, hence $T$ is m-topologisable. 

\begin{proof}
Using barreledness of $\cH$, by Remark \ref{r:opcleq} we have $\cL^*_c(\cH)=\cL^*_b(\cH)=\cL^*(\cH)$. Therefore, in particular, $\fk(x,x)\in\cL^*(\cH)$ is m-topologizable for any $x\in X$. By Proposition 2.17 in \cite{AyGheondea2}, in this case, condition (b2) is automatically satisfied and the theorem follows.
\end{proof}

\begin{remark}
By Corollary \ref{c:quotbar}, the assumption of barreledness of the VH-space $\cH$ in Theorem \ref{t:barvhinvkolmo} can be replaced by the assumption that, all the quotient normed VE-spaces $\cH_p$ are barreled, for $p\in S_0(Z)$ with $S_0(Z)\subseteq S_*(Z)$ any generating family of seminorms.  
\end{remark}

\begin{remark}
Theorem \ref{t:barvhinvkolmo} provides a partial answer to a question raised in \cite{AyGheondea2}: When the condition (b2) in Theorem \ref{t:vhinvkolmo2} is satisfied automatically? Thus, one answer is when the VH-space $\cH$ is a barreled VH-space, or when the quotient normed VE-spaces $\cH_p$, $p\in S_0(Z)$ are barrelled. A full characterization of the condition (b2) remains an open problem. 
\end{remark}

\begin{remark}
It would be interesting to know if the linearisation spaces $\cK$ and/or $\cR$ in Theorem \ref{t:barvhinvkolmo} are barreled VH-spaces in general. 
\end{remark}

\end{document}